\documentclass{amsart}

\usepackage{amssymb, amsmath, esint}

\usepackage{cite}

\usepackage{hyperref}

\renewcommand{\[}{\begin{equation}\begin{aligned}}
\renewcommand{\]}{\end{aligned} \end{equation}}
\newcommand{\ddbar}{\sqrt{-1}\partial \bar\partial }

\usepackage{verbatim}

\newtheorem{thm}{Theorem}
\newtheorem{prop}[thm]{Proposition}

\newtheorem{conj}[thm]{Conjecture}

\theoremstyle{remark}

\theoremstyle{definition}

\author{G\'abor Sz\'ekelyhidi}
\address{Department of Mathematics, Northwestern University, Evanston,
  IL, USA}
\email{gaborsz@northwestern.edu}

\title{Gromov-Hausdorff limits of collapsing Calabi-Yau fibrations}
\date{}

\begin{document}

\begin{abstract}
  We study Calabi-Yau metrics on a projective manifold in K\"ahler
  classes converging to a semiample class given by a fibration. We show that the
  Gromov-Hausdorff limit of the metrics is homeomorphic to the base of
  the fibration and in addition the discriminant locus has Hausdorff
  codimension at least 2. This resolves conjectures of Tosatti. 
\end{abstract}

\maketitle

\section{Introduction}
Let $M$ be a compact K\"ahler manifold with trivial first Chern
class. By Yau's solution of the Calabi conjecture~\cite{Yau78}, every
K\"ahler class on $M$ admits a unique Ricci flat
K\"ahler metric. A natural and important problem, see
e.g. Yau~\cite[Problem 88]{YauProblems}, is to understand how a family 
$\omega_t$ of such metrics can degenerate as the corresponding 
K\"ahler classes, or even the underlying complex manifold varies. See for
example Tosatti~\cite{Tos20} for a recent survey on the topic.

In this paper we are motivated by the setting where the family $[\omega_t]$
converges to a semiample class on $M$. More precisely we assume that $M$ is a
projective manifold with a semiample line bundle $L$, inducing a
holomorphic map $f: M\to X$. We assume that the
cohomology classes of the Ricci flat metrics $\omega_t$ are given by $[\omega_t] =
c_1(L) + tc_1(A)$, where $A$ is an ample line bundle on $M$, and $t\to
0$. It was shown by Tosatti~\cite{Tos09} and Zhang~\cite{Zhang06}
that the diameters of $(M, \omega_t)$ are uniformly bounded as $t\to
0$, and so we can extract compact Gromov-Hausdorff limits along
sequences $t_k\to 0$. A long standing problem is to identify these
limits, the most interesting case being when $\dim X < \dim M$. 

Under the assumptions $f$ is a holomorphic submersion $f: M^\circ \to
X^\circ$ over a Zariski open subset $X^\circ\subset X$.  Moreover,
$X^\circ$ admits a certain canonical K\"ahler metric
$\omega_{can}$ constructed by Song-Tian~\cite{ST12} (see
Section~\ref{sec:collapsing} for more details). 
It was shown by Hein-Tosatti~\cite{HT25} that as $t\to 0$, the forms
$\omega_t$ converge to $\omega_{can}$ in $C^\infty_{loc}(M^\circ)$
after many partial results, see e.g. \cite{GW00, Tos10, TWY18, HT20}.

As for the global
behavior of the metrics,  Song-Tian-Zhang~\cite{STZ19} showed
that the Riemannian manifolds $(M, \omega_t)$ converge to
a compact metric space $(Z, d_Z)$ in the Gromov-Hausdorff
topology, and moreover $(Z,d_Z)$ is the metric completion of the
length metric space $(X^\circ, \omega_{can})$. A basic question that has remained open is
to identify the geometry of this metric completion. The following
conjecture was formulated by Tosatti~\cite{Tos11} (see also
\cite{GTZ13,Tos20}).
\begin{conj}
  The Gromov-Hausdorff limit $(Z, d_Z)$ is homeomorphic to $X$, and
  the set $X\setminus X^\circ$ has Hausdorff codimension at least two
  in $(Z,d_Z)$. 
\end{conj}

The conjecture was shown in several special cases, such as when $X$
has dimension one~\cite{GTZ16}, when
$X$ is smooth and the divisor part of $X\setminus X^\circ$ has simple
normal crossings~\cite{LT23}, when $M$ is hyperk\"ahler~\cite{TZ20},
or when $\dim M = \dim X$, see \cite{Song14}.
The homeomorphism between $Z$ and $X$ was
shown by Song-Tian-Zhang~\cite{STZ19} when $X$ has only orbifold
singularities. Our main result is that the conjecture holds in
general. 

\begin{thm}\label{thm:GHlimit1}
  The Gromov-Hausdorff limit $(Z,d_Z)$ is homeomorphic to $X$. In
  addition we have the Hausdorff dimension bounds $\dim_{\mathcal{H}}
  (X\setminus X^{reg}) \leq 2n-4$ and $\dim_{\mathcal{H}} (X\setminus
  X^{\circ}) \leq 2n-2$ in terms of the metric $d_Z$, where $n =
  \dim_{\mathbb{C}}X$. Here $X^{reg}$ is
  the complex analytically regular subset of $X$. 
\end{thm}

Theorem~\ref{thm:GHlimit1} will follow from a more general statement that we now
describe. Suppose that $(X,\omega_X)$ is an $n$-dimensional compact normal K\"ahler
space, with smooth K\"ahler metric $\omega_X$, and that $\omega =
\omega_X + \ddbar u$ is a (singular) K\"ahler metric on $X$ satisfying
the following conditions:
\begin{enumerate}
\item $u$ is smooth on $X^\circ = X^{reg}\setminus D$ for a divisor $D$,
\item $\omega^n = e^F \omega_X^n$, such that $F\in L^1(\omega_X^n)$ and
  $e^F\in L^p(\omega_X^n)$ for some $p > 1$,
\item $\mathrm{Ric}(\omega) > -A(\omega + \omega_X)$, for some $A >
  0$ on $X^{reg}$. 
\end{enumerate}
Note that beyond normality we are not making assumptions on the
singularities of $X$. This is important since in general $X$ in
Theorem~\ref{thm:GHlimit1} may not be $\mathbb{Q}$-Gorenstein for
instance. We define $\hat{X}$ to be the metric completion of the
metric space $(X^\circ, \omega)$. We will show the following.
\begin{thm}\label{thm:homeo1}
  Suppose that the singular K\"ahler space $(X, \omega)$ satisfies
  the conditions (1),(2),(3) above, and in addition the metric
  completion $\hat{X}$, equipped with the volume measure $\omega^n$,
  is an $RCD(K,2n)$-space for some $K\in\mathbb{R}$, in the sense of
  \cite{AGS14,EKS15}. 
  Then $\hat{X}$ is
  homeomorphic to $X$, and 
  we have the dimension bounds $\dim_{\mathcal{H}} (X\setminus
  X^{reg}) \leq 2n-4$, $\dim_{\mathcal{H}} (X\setminus X^\circ) \leq
  2n-2$ as in Theorem~\ref{thm:GHlimit1}. 
\end{thm}

It is natural to conjecture that the RCD assumption in
Theorem~\ref{thm:homeo1} is actually superfluous. This was conjectured
previously in \cite{SzRCD} for singular K\"ahler-Einstein metrics,
where it
was also shown under the assumption that the K\"ahler-Einstein metric can
be approximated by constant scalar curvature K\"ahler metrics on a
resolution. See also Guo-Song~\cite{GS25}, Fu-Guo-Song~\cite{FGS25}
and Chen-Chiu-Hallgren-Sz\'ekelyhidi-T\^o-Tong~\cite{AIM} for 
further results in this direction.  
\begin{conj}
  Suppose that $(X,\omega)$ satisfies conditions (1), (2), (3)
  above. Then the metric completion $\hat{X}$, equipped with the
  measure $\omega^n$, is an RCD space. 
\end{conj}

In order to deduce Theorem~\ref{thm:GHlimit1} from
Theorem~\ref{thm:homeo1} the main additional ingredient is to show that
the metric completion $\hat{X}$ of $(X^\circ, \omega_{can})$, equipped with the
measure $\omega_{can}^n$, is an $RCD(K, 2n)$
space. Note that since $\hat{X}$ is the Gromov-Hausdorff limit of the
$(M,\omega_t)$, it follows by Ambrosio-Gigli-Savar\'e~\cite{AGS14}
that $\hat{X}$, equipped with the
renormalized limit measure, is 
automatically an $RCD(K, 2m)$ space, where $m=\dim_{\mathbb{C}}(M)$,
but typically $m > n$ (see also Sturm~\cite{Sturm06} and
Lott-Villani~\cite{LV09} for earlier results about $CD$-spaces).
We will follow arguments of
Gross-Tosatti-Zhang~\cite[Lemma 5.2]{GTZ13} and Fu-Guo-Song~\cite[Proof
of Theorem 1.3]{FGS20} to show that the renormalized limit measure on
$\hat{X}$ actually coincides with $\omega_{can}^n$,
which will then lead to the required $RCD(K, 2n)$-property.

We will now describe the proof of Theorem~\ref{thm:homeo1}. The basic
strategy is based on the author's previous work~\cite{SzRCD}, which in
turn uses the techniques of Donaldson-Sun~\cite{DS14}. These methods
were extended in Guo-Song~\cite{GS25} to a more general
setting of K\"ahler spaces with Ricci curvature bounded below, where
the divisor $D$ is empty. The main
new difficulty in Theorem~\ref{thm:homeo1} in comparison with earlier
works is that there could be points in the divisor $D$ that are
regular in the sense of metric tangent cones, even though the metric $\omega$
is in general singular along $D$. As a result, rescalings of the
metric on $\hat{X}$ do not necessarily have good convergence
properties on the regular sets of its tangent cones. This issue will be dealt with
in Proposition~\ref{prop:coords} below. An important ingredient is a
strict positivity result for singular K\"ahler metrics $\omega$
satisfying the conditions (1),(2),(3), that was shown in \cite{AIM},
generalizing previous work in \cite{SzRCD}. The upshot is that
near almost regular points of $\hat{X}$ in the metric sense we can find
smooth Gromov-Hausdorff approximations with Ricci curvature bounded
below. These approximations can be used to apply the results of
\cite{LiuSz1} on the construction of good local holomorphic charts. 

It is worth noting that a similar difficulty appears in the dimension
estimate for $X\setminus X^\circ$, namely that this set could contain
points that are regular in the sense of metric tangent cones. Because
of this, Cheeger-Colding's dimension estimates~\cite{CC1}, and their
extension to the RCD setting by De Philippis-Gigli~\cite{DPG18}, do
not immediately lead to the Hausdorff codimension 2 result. Instead
we  exploit the fact that locally $X\setminus X^\circ$ is cut out
by holomorphic functions.

To close the introduction we briefly mention another setting where
Theorem~\ref{thm:homeo1} can be applied. Here $M$ is a projective
manifold of dimension $m$, whose canonical line bundle $K_M$ is semiample, and we assume
that the Kodaira dimension is $n \in (0, m]$. Given an ample line
bundle $A$ on $M$, La Nave-Tian~\cite{LNT16} showed that there are
unique metrics $\omega_t\in c_1(tA+K_M)$ satisfying the equations
\[ \mathrm{Ric}(\omega_t) = -\omega_t + t\omega_A, \]
for a K\"ahler metric $\omega_A\in c_1(A)$. As above, $K_M$ induces a
holomorphic map $f: M\to X$ which is a submersion over $X^\circ$. It was shown in
Song-Tian-Zhang~\cite{STZ19} that similarly to the collapsing
Calabi-Yau setting considered above, there is a canonical metric
$\omega_{can}$ on $X^\circ$, and $(M, \omega_t)$ converge in the
Gromov-Hausdorff sense to the metric completion $\hat{X}$ of $(X^\circ,
\omega_{can})$as $t\to 0$. Essentially the same arguments as in the proof of
Theorem~\ref{thm:homeo1} imply that $\hat{X}$ is
homeomorphic to $X$, and we have the same Hausdorff dimension bounds
as in Theorem~\ref{thm:GHlimit1}. We will discuss this setting in more
detail at the end of Section~\ref{sec:collapsing}.

\subsection*{Acknowledgements}
I would like to thank Valentino Tosatti, Jian Song and S{\l }awomir Dinew
for helpful discussions. This work was supported in part by NSF grant DMS-2203218.

\section{Homeomorphism with the underlying variety}
Suppose that $(X,\omega)$ is a normal K\"ahler space satisfying the
conditions (1), (2), (3) in the introduction for a divisor $D\subset
X$. In addition, we assume that the metric completion $(\hat{X}, d)$ of
$(X^{reg}\setminus D, \omega)$ is an $RCD(K, 2n)$-space with the
volume form $\omega^n$. Note that by \cite[Theorem 1.1]{AIM} we know
that $\omega > c_0 \omega_X$ for a constant $c_0 > 0$, where
$\omega_X$ is a smooth K\"ahler metric on $X$. Our first goal in this section is
to prove the following, generalizing the author's earlier result in
\cite{SzRCD}.

\begin{thm}\label{thm:homeo2}
  The metric space $\hat{X}$ is homeomorphic to the variety $X$.
\end{thm}

First, observe
that the inclusion map $(X^{reg}\setminus D, \omega) \to
(X, \omega_X)$ is Lipschitz because of the lower bound $\omega > c_0
\omega_X$. It follows that 
this map has a Lipschitz extension
\[ \label{eq:Phidefn} \Phi : (\hat{X}, d_{\hat{X}}) \to (X, \omega_X) \]
to the metric completion, which is surjective by the density of
$X^{reg}\setminus D$ in $X$. To prove Theorem~\ref{thm:homeo2}
we need to show that $\Phi$ is injective.

Following the strategy of \cite{SzRCD}, we show that if $p_1 \not=
p_2$ in $\hat{X}$, with $\Phi(p_1) = \Phi(p_2) = p \in X$, then in a
neighborhood of $p$ we can construct a bounded holomorphic function $f$ on
$X^{reg}\setminus D$ whose extension to $\hat{X}$ (as complex harmonic
functions) has different values
at $p_1, p_2$. This is a contradiction, since $f$ extends continuously
as a holomorphic function to $X$ as well, showing that $\Phi$ is in fact
injective. In comparison with the proof in \cite{SzRCD}
the presence of the divisor $D$ causes substantial new technical
difficulties (see Proposition~\ref{prop:coords}), and note also that here
we are not assuming that the space is projective.

Instead of working with a global line bundle on $X$ as
in \cite{SzRCD}, we will work with a cover of $X$ with Stein
neighborhoods. Specifically, for any point $q\in X$ we can choose a Stein
neighborhood $\mathcal{U}$ of $q$, for instance the intersection of $X$ with a
Euclidean ball under a local embedding of $X$ into $\mathbb{C}^N$. We
then consider $\mathcal{V} = \mathcal{U}\cap (X^{reg}\setminus D)$. We
also choose a smaller neighborhood $q\in \mathcal{U}'\subset\subset
\mathcal{U}$ and let $\mathcal{V}' = \mathcal{V}\cap \mathcal{U}'$. We can
assume that on $\mathcal{V}$ we have $\omega = \ddbar u$ for a bounded
function $u$, and we let $L$
denote the trivial line bundle over $\mathcal{V}$, equipped with the (smooth)
Hermitian metric $h=e^{-u}$. Note that $\mathcal{V}$ admits a complete
K\"ahler metric, using Demailly~\cite{Dem82}, and so we can apply the
following H\"ormander estimate (see \cite[Theorem 4.1]{Dem82})
to the line bundle $P = L^k\otimes K_X^{-1}$
over $\mathcal{V}$, for large $k$. 

\begin{thm}\label{lem:Hormander}
  Let $(P,h_P)$ be a Hermitian holomorphic
  line bundle on a K\"ahler manifold $(M, \omega_M)$, which
  admits a complete K\"ahler metric. Suppose that the curvature
  form of $h_P$ satisfies $\sqrt{-1}F_{h_P} \geq c \omega_M$ for some constant
  $c > 0$. Let $\alpha\in \Omega^{n,1}(P)$ be such that
  $\bar\partial\alpha = 0$. Then there exists $u\in \Omega^{n,0}(P)$
  such that $\bar\partial u = \alpha$, and
  \[ \Vert u\Vert_{L^2}^2 \leq \frac{1}{c} \Vert
    \alpha\Vert^2_{L^2}, \]
  provided the right hand side is finite. 
\end{thm}

We will let $M=\mathcal{V}$ with the metric $\omega_M = k\omega$ for
some $k > 0$. On $P=L^k\otimes K_X^{-1}$ we use the metric
$h^k$ on $L^k$ and $\omega^n$ on $K_X$. 
We have the following estimate for sections of $L^k$, analogous to
\cite[Proposition 19]{SzRCD}. 

\begin{prop}
  Suppose that $f$ is a holomorphic section of $L^k$ over
  $\mathcal{V}$. Then we have
  \[ \sup_{\mathcal{V}'} |f|_{h^k} + |\nabla f|_{h^k, \omega_M} \leq
    K_1 \Vert f\Vert_{L^2(\mathcal{V}, h^k, \omega_M)}, \]
  where $K_1$ does not depend on $k$. 
\end{prop}
\begin{proof}
  The proof is fairly standard, similar to that in \cite{SzRCD}, once
  we extend the section to $\mathcal{U}$. 
  Since $\omega > c_0\omega_X$, and $h = e^{-u}$ for a bounded
  function $u$, a section $f$ of $L^k$ with bounded $L^2$-norm with
  respect to $h^k, \omega$  can be identified with a holomorphic
  function on $\mathcal{V}$ that is $L^2$ with respect to
  $\omega_X$. It follows that first $f$ extends across $D$ to a
  holomorphic function on $\mathcal{U}\cap X^{reg}$. Then using that
  $X$ is normal we can extend $f$ to $\mathcal{U}$. In particular $f$
  is bounded on compact subsets of $\mathcal{U}$. Using the
  gradient estimate for RCD spaces from Jiang~\cite{Jiang14} it also
  follows that $|\nabla f|_{h^k}$ is bounded on compact sets, using
  again that  the metric $h$ is $e^{-u}$ for bounded $u$.

  To obtain the estimates for $|f|_{h^k}, |\nabla f|_{h^k, \omega_M}$,
  note that both quantities satisfy differential inequalities of the
  form $\Delta v \geq -Cv$ in terms of the metric $\omega_M$. We have
  an $r_0 > 0$ such that for any point $x\in \mathcal{V}'$ the ball
  $B_{\omega_M}(x,r_0) \subset \mathcal{U}$, no matter what $k$
  is. Using this, together with the mean value inequality and gradient
  estimate for harmonic functions, we obtain the required bounds for
  $|f|, |\nabla f|$. 
\end{proof}

Our next goal is to show that the almost regular set in $\hat{X}$ in
the metric sense is contained in $X^{reg}$. More precisely, for any
$\epsilon > 0$ we define the $\epsilon$-regular set
$\mathcal{R}_\epsilon(Y)$ in a noncollapsed $2n$-dimensional
RCD space $Y$ to consist of the set of points $p$ that satisfy
\[ \lim_{r\to 0} r^{-2n} \mathrm{vol}(B(p,r)) > \omega_{2n} -
  \epsilon, \]
where $\omega_{2n}$ is the volume of the $2n$-dimensional Euclidean
unit ball. We have the following, similar to \cite[Proposition
24]{SzRCD}.
\begin{prop}\label{prop:epsilonreg}
  There is an $\epsilon_3 > 0$, depending on $(X,\omega)$, such that
  the map $\Phi$ in \eqref{eq:Phidefn} is injective on 
  $\mathcal{R}_{\epsilon_3}(\hat{X})$, and its image is contained in $X^{reg}$.  
\end{prop}
\begin{proof}
  The proof mainly follows along the lines of the proof in
  \cite{SzRCD}, and we will only focus on the differences here. Note 
  that there may be additional metric singular points contained in the
  divisor $D$, so $\mathcal{R}_{\epsilon_3}(\hat{X})$ may be smaller
  than $X^{reg}$, and in addition it is not immediate that $\Phi$ is
  injective over the preimage $\Phi^{-1}(X^{reg})$. 

  Let us define $\Gamma = (X\setminus X^{reg})\cup D$, which is the
  set where $\omega$ is not assumed to be smooth. We suppose that we
  have Stein neighborhoods $\mathcal{U}_k' \subset \mathcal{U}_k$ as
  above, such that the $\mathcal{U}_k'$ cover $X$. In addition we can
  assume that we have bounded holomorphic functions $s_k$ on
  $\mathcal{U}_k$ such that $\Gamma \cap \mathcal{U}_k \subset
  s_k^{-1}(0)$. Let us write $\hat{\mathcal{U}}_k, \hat{\mathcal{U}}_k'$
  for the preimages of $\mathcal{U}_k, \mathcal{U}_k'$ under the map
  $\Phi$. The $s_k$ define bounded complex valued harmonic functions
  on the $\hat{\mathcal{U}}_k$.
  Using $\omega > c_0\omega_X$
  we also have $r_0 > 0$ such that $B(q, r_0)\subset
  \hat{\mathcal{U}}_k$ for any $q\in \hat{\mathcal{U}}_k'$.  

  We can argue exactly like in \cite[Lemma 20]{SzRCD} to show that
  there are constants $c_1, N > 0$ such that
  \[ \int_{B(q, r)} |s_k|^2\, \omega^n \geq c_1 r^N, \]
  for all $r\in (0,r_0)$ and $q\in \hat{\mathcal{U}}_k'$. Using this,
  we obtain the same result as \cite[Lemma 22]{SzRCD}. Namely, there
  are $\epsilon_0, r_0 > 0$, such that whenever $q\in
  \hat{\mathcal{U}}_k'$ and
  \[ d_{GH} (B(q, r_1) , B(0_{\mathbb{R}^{2n}}, r_1)) <
    r_1\epsilon_0, \]
  for some $r_1\in (0,r_0)$, then we have
  \[ \limsup_{r\to 0} \frac{ \fint_{B(q, r)} |s_k|^2\,
      \omega^n}{\fint_{B(q, r/2)} |s_k|^2\, \omega^n} \leq 2^{2N}. \]

  Suppose that we have a sequence of points $p_j \in \hat{X}$, and
  $r_j \to 0$ such that the pointed Gromov-Hausdorff limit of the rescaled 
  sequence $(\hat{X}, r_j^{-1}d_{\hat{X}}, p_j)$ is
  $\mathbb{R}^{2n}$. Without loss of generality we can assume that
  each $p_j$ is in one of our neighborhoods $\hat{\mathcal{U}}_k'$,
  which we denote for simplicity by $\hat{\mathcal{U}}' \subset
  \hat{\mathcal{U}}$. We will show that for sufficiently large $j$ we
  have $\Phi(p_j)\in X^{reg}$. As above, we let $L$ denote the trivial line
  bundle over $\mathcal{V} = \mathcal{U}\setminus \Gamma$,
  with metric $e^{-u}$, where $\ddbar u = \omega$.

  Let us denote by $\Omega_j$ the domains $\mathcal{U}$,
  equipped with the metric $\omega_j = r_j^{-2}\omega$, and the trivial
  line bundle $L_j$ with metric $e^{-r_j^{-2}u}$. Write
  $\hat{\Omega}_j$ for the preimage of $\Omega_j$ in $\hat{X}$, with
  the rescaled metric $r_j^{-1}d_{\hat{X}}$. Let us denote by
  $q_j\in\hat{\Omega}_j$ the points corresponding to the $p_j$. By decreasing the $r_j$,
  we can assume that on $\Omega_j$ the set $\Gamma$ is
  contained in the zero set of a holomorphic function $f_j$ (a
  suitable scaling of one of the $s_k$ above) such that
  $\Vert f_j\Vert_{L^2(B(q_j, 1))} = 1$, and in addition for any $R >
  0$ if $j$ is sufficiently large we have
  \[ \frac{\fint_{B(q_j, r)} |f_j|^2}{\fint_{B(q_j, r/2)} |f_j|^2}
    \leq 2^{2N+1}, \text{ for all } r < R. \]

  Using that the $f_j$ define complex valued harmonic functions on
  $\hat{\Omega}_j$, the above properties imply that up to choosing a
  subsequence we can extract a limit $f_\infty :
  \mathbb{R}^{2n}\to\mathbb{C}$, which is a non-zero harmonic
  function. By further decreasing the $r_j$ we can assume that
  $f_\infty$ is homogeneous. Let us write $\Sigma =
  f_\infty^{-1}(0)$. We will argue somewhat similarly
  to the argument in \cite[Proposition 23]{SzRCD} to show that
  $f_\infty$ is actually a holomorphic function, so that $\Sigma$ has
  codimension two.

  Note that if $V\subset\mathbb{R}^{2n}\setminus
  \Sigma$ is a relatively compact open set, then under the pointed convergence
  of $(\hat{\Omega}_j, q_j)$ to $(\mathbb{R}^{2n}, 0)$, $V$ is the
  limit of subsets $V_j\subset \hat\Omega_j$ where the metric
  $\omega_j$ is a smooth K\"ahler metric with Ricci curvature bounded
  below. In \cite{SzRCD} we worked with K\"ahler-Einstein metrics
  and relied on Anderson's $\epsilon$-regularity theorem, while here
  we can use \cite[Theorem 1.4]{LiuSz1} instead. The latter result implies
  that the complex structures on the $V_j$ converge to a complex
  structure on $V$ with respect to which the Euclidean metric is
  K\"ahler (see \cite[Proof of Proposition 2.5]{LiuSz1}), and for
  which $f_\infty$ is holomorphic. Using this we can 
  argue exactly as in \cite[Proposition 23]{SzRCD} to show that
  $\mathbb{R}^{2n}\setminus \Sigma$ is connected, $f_\infty$ is
  holomorphic for a complex structure on $\mathbb{R}^{2n}$, and
  therefore $\Sigma$ is a complex hypersurface defined by a
  homogeneous function.

  At this point we can closely follow the proof of \cite[Proposition
  3.1]{LiuSz1} to prove the following: there is a $C > 0$, and for 
  any $\zeta > 0$ there is a $K > 0$ (depending on $\zeta$) with the
  following property. For sufficiently large $j$ the line bundle $L_j^m$ for some $m < K$
  admits a holomorphic section $s^{(j)}$ over $\hat{\Omega}_j$
  satisfying $\Vert s^{(j)}\Vert_{L^2} \leq C$ and $\Big| |s^{(j)}(x)| - e^{-m
    d(x, q_j)^2/2}\Big| < \zeta$ for $x\in \hat{\Omega}_j$. Note that $s^{(j)}$
  here is simply a holomorphic function on $\Omega_j$, which naturally
  extends to a complex valued harmonic function on $\hat{\Omega}_j$,
  and whose norm is weighted by
  $e^{-mu}$. 

  The section $s^{(j)}$ corresponds to the constant function $1$
  as a section of the trivial line bundle over $\mathbb{C}^n$ with
  metric $e^{-|z|^2}$. By considering the coordinate functions
  $z_1,\ldots, z_n$ on
  $\mathbb{C}^n$ we can similarly construct holomorphic sections
  $s^{(j)}_1,\ldots, s^{(j)}_n$, and the ratios $s^{(j)}_i / s^{(j)}$
  converge to $z_1, \ldots, z_n$ under the Gromov-Hausdorff
  convergence. It follows from this that for sufficiently large $j$
  the point $\Phi(q_j)$, and so the corresponding $\Phi(p_j)$ as well, lies in
  $X^{reg}$. Using this, and a contradiction argument, it follows that for
  sufficiently small $\epsilon_3$ if $p\in
  \mathcal{R}_{\epsilon_3}(\hat{X})$, then $\Phi(p)\in X^{reg}$.

  We
  can see that $\Phi$ is injective on $\mathcal{R}_{\epsilon_3}$, after
  possibly shrinking $\epsilon_3$, as follows. Suppose that $\Phi(p_1)
  = \Phi(p_2)$. We can apply the argument above to the rescaled
  pointed spaces $(\Omega_j, \omega_j, q_j)$, where $\Omega_j=
  \mathcal{U}$, $\omega_j =
  r_j^{-2}\omega_j$ for a sequence $r_j \to 0$, and $q_j$ corresponding
  to $p_1$. Fixing $\zeta = \frac{1}{10}$, once $j$ is
  sufficiently large, we will find a holomorphic function $s$ such
  that $|s(p_1)|e^{-mu(p_1)} > 9/10$, while $|s(p_2)|e^{-mu(p_2)} <
  2/10$. Without loss of generality we can assume that $u(p_1) \geq
  u(p_2)$, in which case the holomorphic function $s$ on $\mathcal{U}$
  separates $p_1$ and $p_2$. This  means that in fact $\Phi(p_1)\not=
  \Phi(p_2)$. 
\end{proof}

The following is the crucial new ingredient that we need in comparison
with the arguments in \cite{SzRCD}.

\begin{prop}\label{prop:coords}
  Given any $\tau > 0$, there exists an $\epsilon > 0$ depending on
  $(X,\omega)$ and $\tau$ with the following property. Suppose that
  $r < \epsilon^2$, and for $p\in \hat{X}$ 
  \[ \label{eq:d10} d_{GH}(B(p, \epsilon^{-1}r), B_{\mathbb{C}^n}(0, \epsilon^{-1}r))
    < r\epsilon. \]
  Then there is a holomorphic chart $F : B(p,r) \to \mathbb{C}^n$,
  which is an $r\tau$-Gromov-Hausdorff approximation to its
  image. In addition we can write $\omega = \ddbar \phi$ on $B(p,r)$
  with $|\phi - d_p^2| < r^2\tau$, where $d_p$ denotes the distance
  to $p$. 
\end{prop}
\begin{proof}
  If the metric $\omega$ on $B(p, \epsilon^{-1}r)$ were smooth, then
  this would be a direct application of \cite[Theorem
  1.4]{LiuSz1}. In order to apply this result, we will first
  approximate $\omega$ with smooth metrics in the Gromov-Hausdorff
  sense. 

  Note that by Proposition~\ref{prop:epsilonreg}, if $\epsilon$
  is chosen sufficiently small, then \eqref{eq:d10} implies that $B(p,
  \epsilon^{-1}r/2)$ can be identified with a subset of $X^{reg}$. By the assumption for
  $\omega$ we have $\mathrm{Ric}(\omega) > -A(\omega + \omega_X)$, and
  by the lower bound $\omega > c_0\omega_X$ it follows that
  $\mathrm{Ric}(\omega) > -A_1\omega$ for some $A_1$ on $X^{reg}$. Let
  us write $\mathrm{Ric}(\omega) = -A_1\omega + \alpha$. The basic idea
  is to try to regularize $\alpha$ using a sequence of $\alpha_k \geq 0$ that are
  smooth on $X^{reg}$, and then approximate $\omega$ with solutions of
  $\mathrm{Ric}(\omega_k) = -A_1\omega_k + \alpha_k$. This is not
  quite what we achieve, but rather we end up regularizing $\alpha +
  A'\omega$ for a sufficiently large $A'$. A technical difficulty
  arises from the fact that $\alpha$ may not be locally $\ddbar$-exact
 around the singularities of $X$, which is why we end up only regularizing on
  compact subsets of $X^{reg}$. 
  
  Consider a resolution $\pi : Y\to X$, which admits a K\"ahler metric
  $\omega_Y$. We can write the equation $\mathrm{Ric}(\omega) =
  -A_1\omega + \alpha$ as a degenerate Monge-Amp\`ere equation on
  $Y$. To be more precise we will work on $Y\setminus E$, where
  $E\subset Y$ is the exceptional divisor. Throughout we will identify
  $Y\setminus E \cong X^{reg}$. On  $Y\setminus E$ 
  we have $\omega = \omega_X + \ddbar u$, where 
  \[ (\omega_X + \ddbar u)^n = e^{A_1u + F}\omega_X^n. \] 
 It follows that on $Y\setminus E$ we have
 \[ \mathrm{Ric}(\omega) = -A_1\ddbar u - \ddbar F + \mathrm{Ric}(\omega_X), \]
 and so
 \[ \label{eq:ddbF1} \ddbar (-F) = \alpha - A_1\omega_X - \mathrm{Ric}(\omega_X). \]
 Note that $F$ is bounded below, since $u$ is bounded, and $\omega >
 c_0\omega_X$, and in fact $-F$ extends as a quasi-psh function on $Y$
 (see the proof of \cite[Proposition 3.1]{AIM}).  
 
 We have $\omega_X^n = e^g \omega_Y^n$ for a function $g$ given by 
 \[ g = \sum_i a_i \log |s_{E_i}|_{h_i}^2, \]
 where $a_i > 0$, and the $s_{E_i}$ are defining sections of the line bundles corresponding to the components $E_i$ of the exceptional divisor, with smooth Hermitian metrics $h_i$. This implies that on $Y\setminus E$ we have $\mathrm{Ric}(\omega_X) > -A_2 \omega_Y$ for some $A_2 > 0$. At the same time we can find a function $\rho$ on $Y$, such that $\rho$ is bounded above, smooth on $Y\setminus E$, and $\rho\to -\infty$ along $E$, such that for some $A_3 > 0$ we have 
 \[ A_3(\omega_X + \ddbar \rho) > \omega_Y. \]
 In particular this implies that on $Y\setminus E$ we have 
 \[ \label{eq:ddbrho1} A_2A_3\ddbar\rho > -A_2A_3\omega_X - \mathrm{Ric}(\omega_X). \]
 Let us define the functions $F_k$ by 
 \[ -F_k = \max \{ -F,  A_2A_3\rho - k\}. \]
 Note that $-F \leq -F_k \leq C$ for a uniform $C$, since $-F$ and
 $\rho$ are bounded above. 
 Using \eqref{eq:ddbF1} and \eqref{eq:ddbrho1} we get
 \[ \ddbar(-F_k) > -A_4\omega_X - \mathrm{Ric}(\omega_X), \]
 on $Y\setminus E$, for some $A_4 > 0$.
 Since $F_k$ is bounded away from $E$, we can further regularize using
 the technique of Blocki-Kolodziej~\cite{BK07} to obtain $\tilde{F}_k$
 such that on any compact subset of $Y\setminus E$ the $\tilde{F}_k$
 are smooth for sufficiently large $k$. More precisely, suppose that
 $K\subset Y\setminus E$ is compact. For small $s > 0$ depending on
 $K$, we can cover $K$ with small balls
 $B_{\omega_Y}(x_i, s)$, such that on the balls $B_{\omega_Y}(x_i,
 10s) \subset\subset Y\setminus E$ we have K\"ahler potentials $\phi_i$
 for $\omega_Y$ satisfying 
 \[ -100s^2 &<  \phi_i < 100s^2, \\
    \inf_{\partial B_{\omega_Y}(x_i, 10s)} \phi_i &> 50s^2, \\
   \sup_{B_{\omega_Y}(x_i, s)} \phi_i &< -s^2. \]
 For this note that $\phi_i$ can be chosen to be approximately equal
 to $d_{\omega_Y}(x_i, \cdot)^2 - 5s^2$. Let $\delta > 0$ be small, to
 be chosen. In each
 ball $B_{\omega_Y}(x_i, 10s)$ we can use a mollifier to construct a
 smooth function $F_{k,\delta,i}$, satisfying $|F_{k, \delta,i} - F_k| < \delta$,
 and
 \[ \ddbar(-F_{k,\delta,i}) > -A_4\omega_X - \mathrm{Ric}(\omega_X). \]
 Let $F_{k,\delta,i}' = F_{k, \delta,i} + \delta^{1/2}\phi_i$. Then
 \[ \label{eq:Fkdeltalower} \ddbar(-F_{k, \delta,i}') > -A_4\omega_X - \delta^{1/2}\omega_Y -
   \mathrm{Ric}(\omega_X), \text{ on } B_{\omega_Y}(x_i, 10s). \]
 We define $-F_{k,\delta}$ to be a regularized maximum of the function
 $-F_k$ together with all of the $-F_{k,\delta, i}'$.
 For sufficiently small $\delta >
 0$ this will be a well defined function on $Y\setminus E$, since if $y\in \partial
 B_{\omega_Y}(x_i, 10s)$ for some $i$, then necessarily
 \[  -F_{k,\delta, i}'(y) &= ( -F_{k,\delta, i}(y) + F_k(y) )+
   (F_{k,\delta, i}(y) - F_{k,\delta, i}'(y)) - F_k(y) \\
   &< \delta - 50 \delta^{1/2} s^2 - F_k(y) \\
   &< -F_k(y) - \delta, \]
 if $\delta$ is sufficiently small (depending on $s$). In addition
 $-F_{k,\delta}$ is smooth on $K$, since if $y\in B_{\omega_Y}(x_i,
 s)$ for some $i$, then
 \[ -F_{k,\delta, i}'(y) &> -F_{k,\delta, i}(y) + \delta^{1/2} s^2 \\
   &> -F_k(y) - \delta + \delta^{1/2}s^2 > -F_k(y) + \delta \]
 for sufficiently small $\delta$, and the $-F_{k, \delta, i}'$ are
 smooth in $B_{\omega_Y}(x_i, 10s)$. Finally note that outside of the union of the balls
 $B_{\omega_Y}(x_i, 10s)$ we have $F_{k,\delta} = F_k$, so
 \[ \label{eq:Fkdeltalower2} \ddbar(-F_{k, \delta}) > -A_4\omega_X - \mathrm{Ric}(\omega_X),
   \text{ on } Y\setminus \left( E \bigcup B_{\omega_Y}(x_i, 10s)\right). \]

 We can define the sequence of regularizations $\tilde{F}_k$ by
 choosing an exhaustion of $Y\setminus E$ with compact sets $K$ as
 above, and for each $K$ choosing $\delta$ sufficiently small so that
 \eqref{eq:Fkdeltalower} and \eqref{eq:Fkdeltalower2} imply that 
\[ \label{eq:ddbarF2}
 \ddbar(-\tilde{F}_k) > -2A_4\omega_X - \mathrm{Ric}(\omega_X), \]
 At the same time we can ensure that we have  $-C \leq \tilde{F}_k
 \leq F$,  for a uniform $C$, and on
 any compact subset of $\pi^{-1}(X^{reg}\setminus D)$ we have $\tilde{F}_k \to F$ smoothly. 
 
 Define $\omega_k = \pi^*\omega_X + \ddbar u_k$ to be solutions of the Monge-Amp\`ere equations
 \[ \omega_k^n = e^{A_1u_k + \tilde{F}_k}\omega_X^n. \]
 Since $\tilde{F}_k \leq F$, and $e^{pF}\omega_X^n$ is integrable for some $p > 1$, it follows from \cite[Theorem 4.1]{EGZ09} that the equation has a unique solution, and we have a uniform $L^\infty$ bound $|u_k| < C$ independent of $k$. Using the argument in \cite[Proposition 3.1]{AIM}, adapted from an argument in Guedj-Lu~\cite{GL25}, it follows that on any compact subset of $\pi^{-1}(X^{reg}\setminus D)$ the $u_k$ converge smoothly to $u$. 
 
 The Ricci curvature of $\omega_k$ on $Y\setminus E$ satisfies 
 \[ \mathrm{Ric}(\omega_k) = A_1(\omega_X - \omega_k) - \ddbar \tilde{F}_k + \mathrm{Ric}(\omega_X), \]
 and so by \eqref{eq:ddbarF2} we get
 \[ \mathrm{Ric}(\omega_k) \geq - A_1\omega_k - 2A_4\omega_X. \]
 Using \cite[Theorem 1.1]{AIM} it follows that $\omega_k > c_1
 \omega_X$, for some $c_1 > 0$ independent of $k$. The upshot is that
 on $B_{\omega}(p, \epsilon^{-1}r/2)$ we have a sequence of smooth
 K\"ahler metrics $\omega_k$, converging locally smoothly to $\omega$
 on $B_{\omega}(p, \epsilon^{-1}r/2)\setminus D$, satisfying a uniform
 Ricci lower bound $\mathrm{Ric}(\omega_k) > -A\omega_k$, and uniform
 lower bounds $\omega_k > c_1\omega_X$. 

 For simplicity let us write $\Omega = B_\omega(p, \epsilon^{-1/2}r)$.  
We claim that $\Omega$, equipped with the metrics $\omega_k$ 
converges in the Gromov-Hausdorff sense to $\Omega$ equipped with the
metric $\omega$. More precisely 
we show that for $q_1, q_2 \in \Omega$ we have $d_{\omega_k}(q_1,q_2) \to
d_{\omega}(q_1, q_2)$, and this convergence is uniform in $q_1,
q_2$. Recall that since $\omega$ may be singular along $D$, the
distance $d_\omega(q_1,q_2)$ is defined using curves lying in
$B_\omega(p, \epsilon^{-1}r/2) \setminus D$, joining sequences of
points $q_1^j \to q_1$ and $q_2^j\to q_2$.
We can closely follow arguments in Song-Tian-Zhang~\cite{STZ19}
for instance. 

For $\delta > 0$ let us denote by $D_\delta$
the set of points $q\in X$ such that $d_{\omega_X}(q, D) >
\delta$. Let us also choose $\delta_0 > 0$ such that the
$\omega_k$-distance from $\Omega$ to $\partial
B_{\omega}(p, \epsilon^{-1}r)$ is at least $2\delta_0$ for all $k$.
Such $\delta_0$ exists because of the lower bound $\omega_k >
c_1\omega_X$. 
The main ingredient is the following ``almost convexity'' result for the set
$D_\delta\cap \Omega$. We claim that
for any $\delta > 0$ there exists a $\rho > 0$ such that for all $k$ if
$q_1,q_2\in \Omega \cap D_\delta$ and
$d_{\omega_k}(q_1,q_2) < \delta_0$, then 
there is a smooth path $\gamma_k \subset B_\omega(p, \epsilon^{-1}r)\cap
D_\rho$ such that its length with respect to $\omega_k$ satisfies
\[ \ell_{\omega_k}(\gamma_k) \leq d_{\omega_k}(q_1, q_2) + \delta. \]
This closely follows the argument in \cite[Lemma 2.7]{STZ19}, or
Datar~\cite[Proposition 1.1]{Dat16}, based on relative volume
comparison. Note that our assumption means that if
$\ell_{\omega_k}(\gamma_k) < 2\delta_0$, then $\gamma_k$ is contained
in $B_\omega(p, \epsilon^{-1}r/2)$. 

Suppose now that $q_1,q_2\in \Omega$ and
$d_{\omega_k}(q_1, q_2) < \delta_0$. Since we have a uniform
$L^p$-bound for the volume forms of $\omega_k$, we can use the work of
Guo-Phong-Song-Sturm~\cite[Theorem 3.2]{GPSS2} to get a uniform
volume non-collapsing
bound $\mathrm{Vol}_{\omega_k}(B_{\omega_k}(q, r)) \geq \kappa
r^{2\mu}$ for some $\kappa,\mu > 0$, while a similar bound holds for $\omega$ by the
non-collapsed RCD property. Using the $L^p$ bound for the volume forms
of $\omega, \omega_k$ we also know that the $\omega_k$-volume of the
$\omega_X$-tubular neighborhood 
$B_{\omega_X}(D, \delta)$ satisfies
\[ \mathrm{Vol}_{\omega_k}(B_{\omega_X}(D, \delta)) < \Psi(\delta), \]
where $\Psi(\delta)$ denotes a function satisfying $\Psi(\delta) \to
0$ as $\delta \to 0$, and which we may change below, but which is
independent of $k$. The same estimate also
holds for the $\omega$-volume. This implies that for any
$q\in \Omega$, if $\delta > 0$, and we fix $k$, then
we can find some $q' \in D_\delta$ such that $d_{\omega}(q, q') <
\Psi(\delta)$ and $d_{\omega_k}(q,q') < \Psi(\delta)$. We apply this to
$q_1, q_2$ to obtain points $q_1', q_2'$. The almost convexity result
above, together with the smooth convergence of $\omega_k$ to $\omega$
on $D_\rho$ implies that for sufficiently large $k$ we have
\[ |d_{\omega}(q_1, q_2) - d_{\omega_k}(q_1,q_2)| < \Psi(\delta). \]
The claimed Gromov-Hausdorff convergence follows from this. 

We have therefore shown that under the assumption $\eqref{eq:d10}$,
for sufficiently small $\epsilon$, we can realize $B_\omega(p,
\epsilon^{-1/2}r)$, equipped with the metric $\omega$,
as a Gromov-Hausdorff limit of smooth K\"ahler
manifolds with Ricci curvature bounded below. We can then apply
\cite[Theorem 1.4]{LiuSz1} to obtain holomorphic
charts $F_k : B(p_k,r)\to \mathbb{C}^n$, and K\"ahler potentials
$\phi_k$ on balls $B(p_k, r)$ in this sequence of smooth K\"ahler
manifolds. As in \cite[Proposition 2.4]{LiuSz1}, we can pass to a
limit to obtain the required holomorphic chart $F$. We can also pass
the K\"ahler potentials $\phi_k$ to the limit to get a
plurisubharmonic function $\phi$, and the argue as in \cite[Claim 2.1
of Proposition 2.5]{LiuSz1} to see that $\phi$ is a K\"ahler potential
for $\omega$ on $B(p, r)$ satisfying the required bound. 
\end{proof}

Given this result, we can closely follow the arguments in
\cite[Proposition 3.1]{LiuSz1}. The main difference is that in the
setting of noncollapsed RCD spaces the sharp volume estimates of
Cheeger-Jiang-Naber~\cite{CJN21} do not seem to be available in the
literature. Instead, as described in \cite[Proposition 5.1]{LiuSz1},
we can use the approach of Chen-Donaldson-Sun~\cite{CDS2}, which
deals separately with the codimension two singular stratum.

As above, we focus on one of the Stein neighborhoods $\mathcal{U}$,
and $\mathcal{U}'\subset\subset \mathcal{U}$, with the trivial bundle
$L$ equipped with the metric $e^{-u}$. The proof of the following is
analogous to \cite[Proposition 3.1]{LiuSz1}.
\begin{prop}\label{prop:LSz2}
  Let $(V, o)$ be an iterated tangent cone of $\hat{X}$, with the
  property that the set $V\setminus \mathcal{R}_{\epsilon_3}(V)$ has
  capacity zero, for the $\epsilon_3$ in
  Proposition~\ref{prop:epsilonreg}. There is a constant $C > 0$, and
  for any $\zeta > 0$ there are $K, \epsilon > 0$ depending on $\zeta,
  (X, \omega), V$, with the following property. Suppose that $r > 0$
  satisfies $\epsilon^{-1}r < \epsilon$, and for some $p \in
  \hat{\mathcal{U}}'$ we have
  \[ d_{GH}(B(p, \epsilon^{-1}r), B(o, \epsilon^{-1}r)) < \epsilon
    r. \]
  Then for some $m < K$ there is a holomorphic function $s$ on
  $\mathcal{U}$ such that
  \[ \int_{\mathcal{U}} |s|^2 e^{-r^{-2}mu}\, (r^{-2}m \omega)^n < C, \]
  and 
  \[ \Big| |s(z)| e^{-mr^{-2}u} - e^{-mr^{-2}d(p, z)^2/2}\Big| <
    \zeta, \]
  for $z\in \mathcal{U}$. 
\end{prop}

If $V = \mathbb{R}^{2n-2}\times C(S^1_\gamma)$, where $S^1_\gamma$ is
a circle of length $\gamma < 2\pi$, then the singular set
$\mathbb{R}^{2n-2}\times \{0\}$ has capacity zero, so the conclusion
of Proposition~\ref{prop:LSz2} holds. As a result we obtain the
following, whose proof is essentially identical to \cite[Proposition
3.3]{LiuSz1}.

\begin{prop}\label{prop:codim2cone}
  Suppose that we have a sequence $p_j\in \hat{\mathcal{U}}'$ and
  $\lambda_j\to\infty$ such that the rescaled pointed sequence $(\hat{X},
 \lambda_j d_{\hat{X}}, p_j)$ converges to
  $\mathbb{R}^{2n-2}\times C(S^1_\gamma)$ in the pointed
  Gromov-Hausdorff sense. Then for sufficiently large $j$
  the map $\Phi$ in \eqref{eq:Phidefn} is an injection on
  $B(p_j, \lambda_j^{-1})$, its image is in
  $X^{reg}$,  and in  addition
  \[ \lim_{j\to\infty} \int_{B(p_j, \lambda_j^{-1})} n\,
    \mathrm{Ric}(\omega)\wedge \omega^{n-1} = \lambda_j^{2-2n}
    \omega_{2n-2}(2\pi - \gamma), \]
  where $\omega_{2n-2}$ is the volume of the unit ball in
  $\mathbb{R}^{2n-2}$. 
\end{prop}

Given this, we can argue exactly as in \cite[Proposition
5.1]{LiuSz1} to show that the singular set $V\setminus
\mathcal{R}_{\epsilon_3}(V)$ has capacity zero in any iterated tangent
cone of $\hat{X}$. Therefore the conclusion of
Proposition~\ref{prop:LSz2} holds for any iterated tangent cone. As a
consequence we can complete the proof of Theorem~\ref{thm:homeo2}.

\begin{proof}[Proof of Theorem~\ref{thm:homeo2}]
We can argue similarly to the end of the proof of
Proposition~\ref{prop:epsilonreg}. 
Suppose that $\Phi(p_1) = \Phi(p_2)$. Then $p_1, p_2$ are contained in
the same neighborhood $\mathcal{U}' \subset\subset \mathcal{U}$. We can
assume that the K\"ahler potential $u \leq 0$, and $u(p_1) =
0$. Applying Proposition~\ref{prop:LSz2} to the tangent cone at $p_1$
with $\zeta = \frac{1}{10}$, and sufficiently small $r$, we can
construct a holomorphic function $s$ on 
$\mathcal{U}$ such that $s(p_1) > 9/10$, while $s(p_2) < 2/10$. This
means that $\Phi(p_1) \not= \Phi(p_2)$. 
\end{proof}

We will now prove the Hausdorff dimension estimates for $X\setminus
X^{reg}$ and $X\setminus X^\circ$, which will complete the proof of
Theorem~\ref{thm:homeo1}. Note that at this point we know that
$\hat{X}$ is homeomorphic to $X$ and so we identify the two spaces. 
\begin{prop}
  We have $\mathrm{dim}_{\mathcal{H}}(X\setminus X^{reg}) \leq 2n-4$,
  and $\mathrm{dim}_{\mathcal{H}}(X\setminus X^\circ) \leq 2n-2$. 
\end{prop}
\begin{proof}
  Let us first consider $X\setminus
  X^{reg}$. Proposition~\ref{prop:codim2cone} implies that if $p\in
  X\setminus X^{reg}$, then no tangent cone at $p$ can split a factor
  of $\mathbb{R}^{2n-2}$. We claim that if a tangent cone $V$ at $p$
  splits a factor of $\mathbb{R}^{2n-3}$, then it also follows that
  $p\in X^{reg}$. We expect more generally that in analogy with
  Cheeger-Colding-Tian~\cite[Theorem 9.1]{CCT02}, if a tangent cone at
  $p$ splits a factor of $\mathbb{R}^k$ for odd $k$, then in fact it splits
  $\mathbb{R}^{k+1}$, however here we will just
  deal with our special situation where $k=2n-3$.

  Suppose that a tangent cone at $p\in \hat{X}$ splits a factor of
  $\mathbb{R}^{2n-3}$, so is of the form $V = \mathbb{R}^{2n-3}\times
  C(Y)$, where $Y$ is 2-dimensional. 
  Using the techniques of
  Donaldson-Sun~\cite{DS17} and \cite{LiuSz2}, one can show that $V$
  has the structure of a normal affine algebraic variety. We will
  first show that in fact $V$ is smooth as a complex variety.

  Away from
  $\mathbb{R}^{2n-3}\times\{o\}$, where $o\in C(Z)$ is the vertex,
  each tangent cone of $V$ splits at least a factor of
  $\mathbb{R}^{2n-2}$, and so using Proposition~\ref{prop:codim2cone}
  this part of $V$ has the structure of a smooth complex manifold. At
  the same time, we claim that if $q\in \mathbb{R}^{2n-3}\times\{o\}$
  were a singular point, then $V$ would be singular along all of
  $\mathbb{R}^{2n-3}\times\{o\}$. This would contradict that by
  normality of $V$ the singular set should have real codimension at
  least 4. To see this, it is enough to show that the isometries of $V$ given
  by translations in the $\mathbb{R}^{2n-3}$ factor preserve the
  complex structure.

Let  $f : V\to\mathbb{R}$ be the projection to one of the
  $\mathbb{R}$ factors, and let $V' = V \setminus
  (\mathbb{R}^{2n-3}\times \{o\})$. Let $\epsilon > 0$,  $r < \epsilon^2$ and $B(q,
  \epsilon^{-1}r)\subset V'$ satisfying \eqref{eq:d10} as in
  Proposition~\ref{prop:coords}. As in the proof of that proposition,
  we can realize $B(q, r)$ as the Gromov-Hausdorff limit
  of a sequence of balls $B(q_k, \epsilon^{-1} r)$ in (incomplete) K\"ahler
  manifolds $(M_k, g_k)$ with Ricci
  curvature bounded below uniformly, and we can assume that the Ricci
  curvature of $g_k$ is bounded below by $k^{-1}$. Using
  \cite[Theorem 1.4]{LiuSz1} we can find holomorphic coordinates on the
  $B(q_k, r)$, which converge to holomorphic coordinates on $B(q,r)$
  under the Gromov-Hausdorff convergence
  $B(q_k, r)\to B(q,r)$. Using these coordinates we can identify the
  balls $B(q_k, r)$ with $B(q,r)$. Since the
  function $f$ defines a splitting function on $B(q, r)$, it follows
  that we have corresponding harmonic almost splitting functions $f_k$
  on $B(q_k, r)$ converging to $f$, and satisfying
  \[ \label{eq:intD2f} \fint_{B(q_k, r)} |\nabla^2 f_k|^2 \to 0. \]
  We now work in analogy with the argument in \cite[Proposition 15]{LiuSz2}. Let
  $v_k = \nabla f_k$, and $\sigma_{k, t}$ the diffeomorphisms
  generated by $v_k$. Write $v_k^{(1,0)}$ for the $(1,0)$-part of
  $v_k$. Arguing as in \cite{LiuSz2}, we can approximate the $v_k$ with
  holomorphic vector fields $\tilde{v}_k$, generating biholomorphisms
  $\tilde{\sigma}_{k, t}$, and in the limit as $k\to \infty$, the
  limiting diffeomorphisms $\sigma_t$ of the $\sigma_{k,t}$ coincide with the
  biholomorphisms $\tilde{\sigma}_t$ given by the limit of
  $\tilde{\sigma}_{k,t}$.
  At the same time the
  diffeomorphisms $\sigma_t$ are given by translation in the
  $\mathbb{R}$-factor corresponding to the splitting function $f$.

  It follows that $V$ cannot have any singularities, since otherwise
  the singular set would have real codimension at most 3. 
  So, as claimed, $V$ is a smooth complex variety. Let
  $z_1,\ldots, z_n$ be a coordinate chart in a neighborhood of the
  vertex in $V$. Using the H\"ormander $L^2$-estimates we can use
  these to construct a holomorphic chart in a sufficiently small
  neighborhood of $p\in X$, and so $p\in X^{reg}$ as claimed. It
  follows that tangent cones at points $p\in X\setminus X^{reg}$
  can split at most a factor of $\mathbb{R}^{2n-4}$, and so by the
  dimension estimates of De Philippis-Gigli~\cite{DPG18}, the
  Hausdorff dimension of $X\setminus X^{reg}$ is at most $2n-4$.

  Let us now consider the claim that
  $\mathrm{dim}_{\mathcal{H}}(X\setminus X^\circ) \leq 2n-2$. Note
  that $X\setminus X^\circ \subset (D\cap
  \mathcal{R}) \cup  (\hat{X}\setminus
  \mathcal{R})$,  where $\mathcal{R}$ denotes the metric regular set (where the
  tangent cones are $\mathbb{R}^{2n}$). The dimension estimate $\dim_{\mathcal{H}}(\hat{X}\setminus
  \mathcal{R}) \leq 2n-2$ follows from De
  Philippis-Gigli~\cite{DPG18}. Suppose that we have  $\dim_{\mathcal{H}}(D\cap
  \mathcal{R}) > 2n-2$, so that there is some $\ell>
  2n-2$ and a point of density $x\in  (D\cap
  \mathcal{R})$ for the $\ell$-dimensional
  Hausdorff measure. As in the proof of
  Proposition~\ref{prop:epsilonreg} we can now exploit that $D$ is cut out
  by a holomorphic function of finite order of vanishing. Let $r_k \to
  0$, and consider the rescaled balls $r_k^{-1}B(x, r_k)$ converging
  to the Euclidean unit ball $B(0, 1)$ in the tangent cone at $x$. We
  can assume that $D\cap B(x,r_1) = s^{-1}(0)$ for a holomorphic
  function $s$ on $B(x,r_1)$. Let us denote by $s_k$ the holomorphic
  function $s$ normalized to have $L^2$-norm 1 on the rescaled ball
  $r_k^{-1}B(x, r_k)$. As in the proof of Proposition~\ref{prop:epsilonreg} we can assume that
  up to choosing a subsequence, after suitable normalizations the
  $s_k$ converge to a non-zero holomorphic function $s_\infty$ on
  $B(0,1)\subset \mathbb{C}^n$, and under the Gromov-Hausdorff
  convergence the points of $D\cap B(x,r_k)$ converge to the zero set
  $s_{\infty}^{-1}(0)$. Using that $x$ was a point of density for the
  $\ell$-dimensional Hausdorff measure, it follows that $s_{\infty}^{-1}(0)$ has Hausdorff
  dimension at least $\ell$, if the $r_k$ are chosen suitably. But this contradicts that the zero set of
  a non-zero holomorphic function has Hausdorff dimension at most $2n-2$.
\end{proof}

\section{Collapsing Calabi-Yau metrics}\label{sec:collapsing}
In this section our goal is to prove Theorem~\ref{thm:GHlimit1},
by showing that the setting of collapsing
Calabi-Yau metrics fits into the
assumptions of Theorem~\ref{thm:homeo1}. Recall that $M$ is a (smooth)
projective manifold of dimension $m$ with trivial first Chern class,
together with a semiample line bundle $L$, giving rise
to a holomorphic map $f : M\to X$, where $\dim X = n$. We let $A$ denote an ample line
bundle on $M$, and consider the Ricci flat metrics $\omega_t \in
c_1(L) + t c_1(A)$ for $t > 0$, whose existence is provided by Yau's
Theorem~\cite{Yau78}.

More precisely, we let $\Omega$ be a smooth
volume form on $M$ which satisfies $\mathrm{Ric}(\Omega)=0$, and
$\int_M \Omega = 1$. Choosing suitable metrics $\omega_M\in c_1(A),
\omega_X\in c_1(L)$ on
$M, X$, we can write $\omega_t = f^*\omega_X + t\omega_M + \ddbar\psi_t$,
satisfying the Monge-Amp\`ere equation
\[ \omega_t^m = t^{m-n} c_t \Omega. \]
Integrating, we find that the normalizing constants $c_t$ satisfy
\[c_0 := \lim_{t\to 0} c_t = \binom{m}{n} c_1(L)^n \cup
  c_1(A)^{m-n}. \]

There is a complex subvariety $B\subset X$, such that $f$ is a
holomorphic submersion over $X^\circ = X \setminus B$, and $X^\circ$
is smooth. It was shown in Song-Tian~\cite{ST12} that $X$ admits a
twisted K\"ahler-Einstein metric $\omega_{can} = \omega_X + \ddbar u$
with $u\in L^\infty$, 
satisfying
\[ \mathrm{Ric}(\omega_{can}) = \omega_{WP} \]
on $X^\circ$.  Here $\omega_{WP} \geq 0$ is the Weil-Petersson metric given
by the variation of the Calabi-Yau fibers of $f$ over $X^\circ$. On $X^\circ$ the metric
$\omega_{can}$ satisfies the Monge-Amp\`ere equation
\[ \label{eq:canMA} \omega_{can}^n = d_1 f_*\Omega = e^F \omega_X^n, \]
where $d_1 = \int_X \omega_X^n$. 
It is shown \cite[Lemma 2.3]{GS21} that $F$ is bounded below, and
from \cite[Proposition 3.2]{ST12} we have that $e^F \in
L^p(\omega_X^n)$ for some $p > 1$ (note the differences between
our definition for $F$ and the ones used in these references).
It follows that on $X^{reg}$ we can
locally write $\mathrm{Ric}(\omega_{can}) = \ddbar v$, where $v\in
L^1_{loc}$, and $v$ is locally bounded from above. Since
$\mathrm{Ric}(\omega_{can}) \geq 0$ on $X^{\circ}$, it follows from
this that $\mathrm{Ric}(\omega_{can}) \geq 0$ on $X^{reg}$ in the
distributional sense. As a consequence, $(X, \omega_{can})$ satisfies
the conditions (1), (2), (3) in Theorem~\ref{thm:homeo1}, taking $D$
to be a divisor containing the subvariety $B$. 

In order to prove Theorem~\ref{thm:GHlimit1}, it remains to show that
the metric completion $\hat{X}$ is an $RCD(0, 2n)$-space. Note
that $\hat{X}$ is an almost smooth metric measure space in the sense
of \cite[Definition 5]{SzRCD}, which is a slight modification of the
notion due to Honda~\cite{Honda18}. Just as in \cite[Proposition
9]{SzRCD}, by using the work of Honda~\cite{Honda18}, and
Guo-Phong-Song-Sturm~\cite{GPSS2}, we have the following.

\begin{prop}\label{prop:gradestimate}
  Suppose that the $W^{1,2}$-eigenfunctions of the Laplacian on
  $(X^\circ, \omega_{can})$ are Lipschitz. Then $\hat{X}$ is a
  (non-collapsed) $RCD(0, 2n)$-space. 
\end{prop}

In order to show that the eigenfunctions are Lipschitz, we will
exploit the fact that the metric space $\hat{X}$ is the
Gromov-Hausdorff limit of the Ricci flat manifolds $(M, \omega_t)$ as
$t\to 0$, according to Song-Tian-Zhang~\cite[Proposition
3.1]{STZ19}. As shown by Cheeger-Colding~\cite{CC1}, the
Gromov-Hausdorff limit can be equipped with a renormalized limit
measure $\nu$ which we now recall.

We fix basepoints $p_t\in (M,
\omega_t)$, and $p\in \hat{X}$, such that $p_t \to p$ under the
Gromov-Hausdorff convergence. Note that we can take $p\in X^\circ$,
and $p_t \in f^{-1}(p)$ to be independent of $t$, but we write $p_t$
to emphasize that we use the varying metrics $\omega_t$. Then for any $x\in
(M, \omega_t)$ and $r > 0$ we define the renormalized volume function to be
\[ \underline{V}(x,r) =
  \frac{\mathrm{Vol}(B_{\omega_t}(x,r))}{\mathrm{Vol}(B_{\omega_t}(p_t,
    1))}. \]
By \cite[Theorem 1.6]{CC1} we have a continuous function
$\underline{V}_\infty : \hat{X}\times \mathbb{R}_+ \to \mathbb{R}_+$
such that along a sequence $t_j\to 0$ we have the following: for every
$q_j \in (M, \omega_{t_j})$ with $q_j \to x\in \hat{X}$, and every $R
> 0$, we have
\[ \underline{V}(q_j, R) \to \underline{V}_\infty(x, R). \]
By \cite[Theorem 1.10]{CC1} the renormalized limit measure $\nu$ on
$\hat{X}$ is then uniquely defined by the property that $\nu(B(x, R))
= \underline{V}_\infty(x, R)$ for all $x\in \hat{X}$ and $R > 0$.

We next show that $\nu = v \omega_{can}^n$ for a constant $v > 0$. For this we can
follow arguments similar to Gross-Tosatti-Zhang~\cite[Lemma 5.2]{GTZ13} and
Fu-Guo-Song~\cite[Proof of Theorem 1.3]{FGS20}. 

\begin{prop}\label{prop:nulimit}
Up to normalizing the measure, the metric measure space $(\hat{X},
d_{\hat{X}}, \omega_{can}^n)$ is the measured Gromov-Hausdorff
limit of the $(M, d_{\omega_t}, \omega_t^m)$.
\end{prop}
\begin{proof}
To show that $\nu = v\omega_{can}^n$ for a constant $v > 0$, it
is enough to show the following:
\begin{itemize}
  \item[(1)] For any compact set $K\subset X^\circ$ there is an $r_K >
    0$ such that
    \[ \nu(B(x, r)) = v\omega_{can}^n(B(x,r))\, \quad \text{ for
      } x\in K, r < r_K. \]
  \item[(2)] We have $\nu(\hat{X}) = v\omega_{can}^n(X)$.
  \end{itemize}

  Indeed, from (1) it follows that for any compact set $K\subset
  X^\circ$ we have $\nu(K) = v \omega_{can}^n(K)$. At the same
  time we also have that the measure $\nu$ on $\hat{X}$ is obtained by
  extending its restriction to $X^\circ$ trivially. This is because
  for any $\epsilon > 0$ we can find a compact set $K\subset X^\circ$
  such that $\omega_{can}^n(K) > \omega_{can}^n(X) -
  \epsilon$. The conditions (1) and (2) above imply that $\nu(K) >
  \nu(\hat{X}) - v\epsilon$.

  To see the condition (1) above, we use the $C^0$-convergence of the
  metrics $\omega_t$ to $f^*\omega_{can}$ on $f^{-1}(K)$ for any
  compact $K\subset X^\circ$ (see \cite{TWY18}). For any $x\in
  X^\circ$, we can choose $q_j = q \in f^{-1}(x)$ so that $q_j \to x$
  under the Gromov-Hausdorff convergence. Using the $C^0$ convergence
  above, we can find an $r_0 > 0$ and a function $h(t)\to 0$ such that
  for any $r < r_0$ we have
  \[ \label{eq:Bcontain} f^{-1}(B_{\omega_{can}}(x, r-h(t_j))) \subset
    B_{\omega_{t_j}}(q_j, r) \subset f^{-1}(B_{\omega_{can}}(x, r +
    h(t_j))). \] 
  For any $R > 0$ we also have 
  \[ \mathrm{Vol}(f^{-1}(B(x, R), \omega_t)) &=
    \int_{f^{-1}(B(x,R))} \omega_t^m \\
    &= t^{m-n}c_t \int_{f^{-1}(B(x,R))} \Omega \\
    &= t^{m-n} c_t \int_{B(x,R)} f_*\Omega. \]
  It follows that
  \[ \lim_{t\to 0} t^{n-m} \mathrm{Vol}(f^{-1}(B(x,R)), \omega_t)
    = c_0 \int_{B(x,R)} f_*\Omega = c_0 d_1^{-1} \mathrm{Vol}(B(x,R),
    \omega_{can}). \]
  Note that the ball $B(x,R)$ denotes the metric ball of radius $R$ in the metric
  space $\hat{X}$. 
  It follows from this, together with \eqref{eq:Bcontain}, that for $r
  < r_0$ we have
  \[ \label{eq:limVB1} \lim_{j\to\infty} \frac{\mathrm{Vol}(B(q_j, r),
      \omega_{t_j})}{t_j^{m-n}} = c_0d_1^{-1} \mathrm{Vol}(B(x,r),
    \omega_{can}). \]
  Note that the diameters of $(M, \omega_t)$ are uniformly bounded
  from above by a constant $D > 0$  (see Tosatti~\cite{Tos09},  Zhang~\cite{Zhang06}), and at
  the same time 
  \[ \frac{\mathrm{Vol}(M, \omega_t)}{t^{m-n}} \to c_0, \]
  as $t\to 0$. It follows from volume comparison that up to choosing a
  subsequence we have
  \[ \label{eq:volcomp} \frac{\mathrm{Vol}(B(p_{t_j}, 1))}{\mathrm{Vol}(M, \omega_t)} \to
    d_2\]
  for some $d_2 > 0$. Using \eqref{eq:limVB1}, this implies that for
  $r < r_0$ we have
  \[ \nu(B(x,r)) = \lim_{j\to\infty} \frac{\mathrm{Vol}(B(q_j, r),
      \omega_{t_j})}{\mathrm{Vol}(B(p_{t_j}, 1))} = \frac{1}{d_1d_2}
    \mathrm{Vol}(B(x,r), \omega_{can}). \]
  In addition, using \eqref{eq:volcomp} and \eqref{eq:canMA}, we have
  \[ \nu(\hat{X}) = \lim_{j\to\infty} \frac{\mathrm{Vol}(B(q_j, D),
      \omega_{t_j})}{\mathrm{Vol}(B(p_{t_j}, 1))} = \frac{1}{d_2} =
    \frac{1}{d_1d_2} \mathrm{Vol}(M, \omega_{can}^n). \]
  This completes the proof of properties (1) and (2) above, with $v=
  (d_1d_2)^{-1}$. 
\end{proof}

As a consequence we can complete the proof of
Theorem~\ref{thm:GHlimit1}.
\begin{proof}[Proof of Theorem~\ref{thm:GHlimit1}]
As shown by Song-Tian-Zhang~\cite{STZ19}, the Gromov-Hausdorff limit
$(Z, d_Z)$ coincides with the metric completion of $(X^\circ,
\omega_{can})$. We have seen above that $(X, \omega_{can})$ satisfies
the conditions (1), (2), (3) in the introduction. In addition $(Z,
d_Z)$, equipped with the measure $\omega_{can}^n$ is a measured
Gromov-Hausdorff limit of the Ricci flat manifolds $(M_t, \omega_t)$
as $t\to 0$ by Proposition~\ref{prop:nulimit}. It follows by
Cheeger-Colding~\cite{CC3} that the $W^{1,2}$ eigenfunctions of the Laplacian
on $(X^\circ, \omega_{can})$ are Lipschitz, and therefore by
Proposition~\ref{prop:gradestimate} the space $(Z, d_Z,
\omega_{can}^n)$ is an $RCD(0, 2n)$-space. It follows that we can
apply Theorem~\ref{thm:homeo1}, completing the proof. 
\end{proof}

We now consider a closely related setting, namely the continuity
method introduced by La Nave-Tian~\cite{LNT16}. Suppose that $M$ is
projective, of dimension $m$, with semi ample canonical bundle
$K_M$. Let $X$ denote the canonical model of $M$, and $f: M\to X$ the
map induced by a power of $K_M$. Write $n = \dim X$, and assume
that $n\in (0,m]$. Similarly to the above, $f$ is a holomorphic
submersion over $X^\circ = X\setminus B$ for a subvariety $B$.
Let $A$ denote an ample line bundle
on $M$, and $\omega_M\in c_1(A)$ a metric. By \cite{LNT16}, for any
$t\in (0,1]$ we can find a unique $\omega_t\in c_1(K_X) + tc_1(A)$
satisfying the equation
\[ \label{eq:LNT} \mathrm{Ric}(\omega_t) = -\omega_t + t\omega_M. \]
To write this in terms of a Monge-Amp\`ere equation, we can choose a
smooth volume form $\Omega$ on $M$ such that $\mathrm{Ric}(\Omega) =
-f^*\omega_X$ for a smooth K\"ahler metric on $X$. We can then write
\[ \omega_t = f^*\omega_X + t\omega_M + \ddbar \psi_t, \]
where the $\psi_t$ satisfy the equations
\[ \omega_t^m = t^{m-n} e^{\psi_t}\Omega. \]
It was shown in Fu-Guo-Song~\cite{FGS20} that the diameter of $(M,
\omega_t)$ is uniformly bounded, so along a subsequence the $(M,
\omega_t)$ converge in the Gromov-Hausdorff sense to a compact metric
space $(Z,d_Z)$ as $t\to 0$. From Song-Tian-Zhang~\cite[Proposition
2.2]{STZ19} we know that, similarly to the collapsing Calabi-Yau
setting,  $(Z,d_Z)$ is isometric to the metric
completion of $(X^\circ, \omega_{can})$, where $\omega_{can}$ is a
twisted K\"ahler-Einstein metric constructed by Song-Tian~\cite{ST12},
satisfying
\[ \mathrm{Ric}(\omega_{can}) = -\omega_{can} + \omega_{WP}. \]
As before, $\omega_{WP} \geq 0$ on $X^{\circ}$, determined by the
variation of the Calabi-Yau fibers of $f$. Similarly to the above,
$\omega_{can} = \omega_X + \ddbar\psi$, where $\psi$ is a bounded
solution of the Monge-Amp\`ere equation
\[ \omega_{can}^n = d_1 e^\psi f_*\Omega, \]
and $f^*\psi$ is the uniform limit of the $\psi_t$. The arguments from
the collapsing Calabi-Yau setting can easily be modified to obtain the
following.

\begin{thm}
  The metric completion $\hat{X}$ of $(X^\circ, \omega_{can})$ defines
  an $RCD(-1, 2n)$-space, and satisfies Conditions (1), (2), (3)
  in Theorem~\ref{thm:homeo1}. As a consequence $\hat{X}$ is
  homeomorphic to $X$, and satisfies the dimension estimates
  $\dim_{\mathcal{H}}(X \setminus X^{reg}) \leq 2n-4$, and
  $\dim_{\mathcal{H}}(X\setminus X^\circ) \leq 2n -2$. 
\end{thm}

\bibliography{mybib}{}
\bibliographystyle{plain}

\end{document}